\documentclass[a4paper]{article}
\usepackage{amssymb}
\usepackage{amsthm}
\usepackage{latexsym}
\usepackage{amsmath}
\usepackage{color}

\usepackage{comment}
\includecomment{discuss}
\excludecomment{discuss}

\newif\ifcol
\coltrue
\ifcol
\newcommand{\colorr}{\color[rgb]{0.8,0,0}}
\newcommand{\colorg}{\color[rgb]{0,0.5,0}}

\else
\newcommand{\colorr}{\color{black}}
\newcommand{\colorg}{\color{black}}
\fi

\newtheorem{lemma}{Lemma}

\newtheorem{theorem}{Theorem}

\newtheorem{example}{Example}
\newtheorem{corollary}{Corollary}

\begin{document}
\title{
On Asymptotic Properties of Bayes Type Estimators\\ with General Loss Functions}
\author{Teppei Ogihara$^*$\\
$*$ 
\begin{small}Center for the Study of Finance and Insurance, Osaka University\end{small}\\
\begin{small}Japan Science and Technology Agency, CREST\end{small}\\
\begin{small}1-3 Machikaneyama-cho, Toyonaka, Osaka 560-8531, Japan. \end{small}\\
\begin{small}e-mail: ogihara@sigmath.es.osaka-u.ac.jp\end{small}
}
\date{}
\maketitle

\noindent
{\bf Abstract.}
We study asymptotic behaviors of Bayes type estimators 
and give sufficient conditions to obtain asymptotic limit distribution of estimation error.
We assume polynomial type large deviation inequalities and 
prove asymptotic equivalence of estimation error of Bayes type estimator and that of M-estimator
by the virtue of Ibragimov-Has'minskii's theory.
The results can be applied to several statistical models of diffusion processes and jump diffusion processes.
In this paper, we focus on application to a statistical model of an ergodic diffusion process
and give asymptotic normality and convergence of moments of the Bayes type estimator with a general loss function.  


\section{Introduction}

The theory of random fields of likelihood ratios is a powerful tool
to investigate asymptotic behaviors of Bayes type estimators.
This theory is initiated by Ibragimov and Has'minskii \cite{ibr-has01,ibr-has02,ibr-has03}
and applied to statistical models of regular i.i.d. observations and white Gaussian noise models. 
After that, Kutoyants applied Ibragimov-Has'minskii's theory to some statistical models including models of diffusion type processes and point processes. 
See Kutoyants \cite{kut01,kut02} for the details. 
Moreover, Yoshida \cite{yos06,yos05} introduced polynomial type large deviation inequalities
and gave a scheme to obtain asymptotic properties of the M-estimator and the Bayes type estimator
under some moment conditions of a contrast function and its derivatives. 
\begin{discuss}
{\colorr Ibragimovのスキームがexponential deviationであることをコメントするか}
\end{discuss}
This scheme can be applied to many classes of statistical models and gives consistency, asymptotic (mixed) normality and convergence of moments 
of quasi-maximum likelihood estimators and Bayes type estimators. 
See Yoshida \cite{yos06,yos05} for an application to statistical models of ergodic diffusion processes,
Ogihara and Yoshida \cite{ogi-yos01} for models of ergodic jump diffusion processes,
Masuda \cite{mas03} for models of Ornstein-Uhlenbeck processes driven by heavy-tailed symmetric L\'evy processes,
Uchida and Yoshida \cite{uch-yos02} for models of diffusion processes observed in a fixed interval,
Ogihara and Yoshida \cite{ogi-yos02} for models of diffusion processes with nonsynchronous observations.

One of the most important motivations to study quasi-maximum likelihood estimators and Bayes type estimators
is that these estimators are asymptotically efficient in several models.
For statistical models of regular i.i.d. observations, we obtain minimax theorems for estimation errors
and hence can define asymptotic efficiency of estimators.
Since the maximum likelihood estimator and the Bayes estimator attain this bound, these estimators are asymptotically efficient.
See Ibragimov and Has'minskii \cite{ibr-has03}.
We also have asymptotic efficiency of quasi-maximum likelihood estimators and Bayes type estimators for some statistical models of diffusion processes with discrete observations.
Jeganathan \cite{jeg} extended the results of minimax theorems to statistical models satisfying the local asymptotic mixed normality (LAMN) property.
Moreover, Gobet \cite{gob} proved the LAMN property for models of diffusion processes observed in a fixed interval and
the estimators proposed in Genon-Catalot and Jacod \cite{gen-jac93} have the asymptotic minimal variance. 
Gobet \cite{gob02} proved LAN property for statistical models of ergodic diffusion processes, 
\begin{discuss}
{\colorr Yoshida (2006,2011)の推定量(他にあるか?)がefficientとなっている?}
\end{discuss}
and Ogihara \cite{ogi01} gives the LAMN property and asymptotic efficiency of the quasi-maximum likelihood estimator
and the Bayes type estimator proposed in Ogihara and Yoshida \cite{ogi-yos02} for models of diffusion processes with nonsynchronous observations in a fixed interval.

Yoshida \cite{yos06,yos05} applied the results of polynomial type large deviation inequalities to the Bayes type estimator for the quadratic loss function
and obtained asymptotic properties of the estimator. 
The Bayes type estimator for the quadratic loss function can be obtained as a ratio of certain integrals with respect to the parameter, 
and hence can be specified asymptotic behaviors by using polynomial type large deviation inequalities.
On the other hand, Ibragimov and Has'minskii \cite{ibr-has03} treated a wider class of loss functions.
Though their results are for models for i.i.d. observations, we can apply their ideas to models satisfying polynomial type large deviation inequalities,
and can prove asymptotic properties of Bayes type estimators for general loss functions, 
which is the subject of this paper.

In this paper, we prove asymptotic equivalence of the estimation error of the Bayes type estimator and that of the M-estimator.
Thus we obtain the asymptotic distribution of the estimation error of the Bayes type estimator if we have an asymptotic distribution of the M-estimator.
In particular, we see that the asymptotic distribution for the Bayes type estimator does not depend on loss functions.
These results can be applied to models of ergodic diffusion processes, diffusion processes observed in a fixed interval,
ergodic jump diffusion processes and diffusion processes with nonsynchronous observation,
and we obtain asymptotic (mixed) normality and convergence of moments for Bayes type estimators for general loss functions.
Convergence of moments is important when we study the asymptotic expansion of estimators and information criteria.
We focus on an application to models of ergodic diffusion processes in this paper.

This paper is organized as follows.
Section \ref{main-results-section} presents the theories of random fields of likelihood ratio and polynomial type large deviation inequalities,
and we state our main results.
Section \ref{ergodic-diffusion-section} is devoted to an application of main results to statistical models of ergodic diffusion processes.
The proofs of main results are in Section \ref{proofs-section}.

\section{Main results}\label{main-results-section}

We first introduce Ibragimov-Has'minskii's theory of random fields of likelihood ratios.
For computational efficiency, it is reasonable to construct estimators separately for certain subspaces in some statistical models, as seen in Uchida and Yoshida \cite{uch-yos01} and Yoshida \cite{yos05}.
Therefore, we define our model so that it contains these situations. 

Let $K\in\mathbb{N}$, $1\leq k\leq K$, the parameter space $\Theta_k\subset \mathbb{R}^{d_k}$ be a bounded open set $(1\leq k\leq K)$
and $\Theta:=\Theta_1\times \Theta_2\times \cdots \times \Theta_K\subset \mathbb{R}^d$, where $d=\sum_{k=1}^Kd_k$.
If $K\geq 2$, we assume $\Theta_k$ is a convex set for $1\leq k\leq K$.
\begin{discuss}
{\colorr $K\geq 2$の時は, $H_T$に$\underline{\tilde{\theta}}_{k-1}$を入れたり, $\overline{\theta^{\star}}_{k+1}$と$\overline{\hat{\theta}}_{k+1}$をつなぐ線分を考えたりするので, convexityが必要.}
\end{discuss}
Let $(\mathcal{X},\mathcal{A},\{P_{\theta}\}_{\theta})$ be a statistical experiment.
\begin{discuss}
{\colorr that is, $(\mathcal{X}, \mathcal{A})$ be a measurable space and $\{P_{\theta}\}_{\theta\in\Theta}$ is a family of probability measures on $(\mathcal{X},\mathcal{A})$.}
\end{discuss}
Let a random field $H_T:\Theta \times \mathcal{X} \to \mathbb{R}$ be a $C^3$ function with respect to $\theta$ 
and continuously extended as a function on ${\rm clos}(\Theta)\times \mathcal{X}$ for $T>0$, where ${\rm clos}(\Theta)$ represents the closure of $\Theta$.

Let $\overline{\Theta}_k=\Theta_k\times \Theta_{k+1}\times \cdots \times \Theta_K$, 
$\overline{\theta}_k=(\theta_k,\theta_{k+1},\cdots,\theta_K)$ and $\underline{\theta}_k=(\theta_1,\cdots, \theta_k)$ 
for any value $\theta=(\theta_1,\cdots, \theta_K)\in\Theta$,
$a^k_T\in {\rm GL}(d_k)$, 
$b^k_T=(\lambda_{\min}((a^k_T)^{\top}a^k_T))^{-1}\to\infty$,
$\lambda_{\max}((a^k_T)^{\top}a^k_T)\leq C_1(b^k_T)^{-1}$,
where $\top$ represents transpose of a matrix and $\lambda_{\max}(A)$ and $\lambda_{\min}(A)$ represent the maximum and the minimum of eigenvalues of a matrix $A$, respectively.

The theory of Ibragimov-Has'minskii, Kutoyants and Yoshida works on a random field $Z^k_T$ defined by 
\begin{equation*}
Z^k_T(u_k;\underline{\theta}_{k-1},\theta^{\ast}_k,\overline{\theta}_{k+1})=\exp\{H_T(\underline{\theta}_{k-1}, \theta^{\ast}_k+a^k_Tu_k,\overline{\theta}_{k+1})-H_T(\underline{\theta}_{k-1},\theta^{\ast}_k,\overline{\theta}_{k+1})\}.
\end{equation*}
If $H_T$ is a likelihood function of i.i.d. observations, $Z^k_T$ is the original random field of likelihood ratio in Chapter I of Ibragimov and Has'minskii \cite{ibr-has03}.
On the other hand, Yoshida \cite{yos05} worked on $Z^k_T$ when $H_T$ is a general function and studied asymptotic properties 
of the M-estimator and the Bayes type estimator defined by $H_T$ when $T\to\infty$.
M-estimator $\hat{\theta}$ is a random variable defined by  
\begin{equation*}
\hat{\theta}=\hat{\theta}_T=(\hat{\theta}^1_T,\cdots, \hat{\theta}^K_T)={\rm argmax}_{\theta\in {\rm clos}(\Theta)}H_T(\theta).
\end{equation*}
If $H_T$ is the `{\it true}' likelihood function of the statistical model, $\hat{\theta}_T$ is the maximum likelihood estimator,
and if $H_T$ is a `{\it quasi}'-likelihood function, $\hat{\theta}_T$ is called a quasi-maximum likelihood estimator.

To define a Bayes type estimator, we consider a class of loss functions introduced in Section 1.2. of Ibragimov and Has'minskii \cite{ibr-has03}.
For $p>0$ and $1\leq k\leq K$, let ${\bf W}_{p,k}$ be a set of functions $w_k: \mathbb{R}^{d_k} \to [0,\infty)$ satisfying following four properties:
\begin{enumerate}
\item $w_k(0)=0$ and $w_k(u_k)$ is continuous at $u_k=0$ but not identically $0$.
\item $w_k(u_k)=w_k(-u_k)$ for any $u\in\mathbb{R}^{d_k}$.
\item The sets $\{u_k;w_k(u_k)<c\}$ are convex sets for all $c>0$ and are bounded for all $c>0$ sufficiently small.
\item There exists a constant $C>0$ such that $w_k(u_k)\leq C(1+|u_k|^p)$ for $u_k\in\mathbb{R}^{d_k}$. 
\end{enumerate}
Let ${\bf W}_p=\{(w_1,\cdots,w_K);w_k\in{\bf W}_{p,k} \ (1\leq k\leq K)\}$ and ${\bf W}=\cup_{p>0}{\bf W}_p$.

The following is examples of loss functions in Section 1.2. of Ibragimov and Has'minskii \cite{ibr-has03}.
\begin{example}\label{loss-function-ex}
\begin{enumerate} 
\item Let $p>0$ and $w_k(u_k)=|u_k|^p$ for $u_k\in \mathbb{R}^{d_k}$. Then we can easily see $w_k\in{\bf W}_{p,k}$.
\item Let $A$ be a centrally-symmetric bounded convex set such that $0$ is in the interior of $A$ and $w_k:\mathbb{R}^{d_k}\to [0,\infty)$ be defined by
\begin{equation*}
w_k(u_k)=\left\{
\begin{array}{ll}
0, & {\rm if} \  u_k\in A, \\
1, & {\rm if} \  u_k\not\in A.
\end{array}
\right.
\end{equation*}
Then we obviously have $w_k\in\cap_{p>0}{\bf W}_{p,k}$.
\end{enumerate}
\end{example}

Let $\mathcal{K}$ be a compact subset of $\Theta$, 
\begin{discuss}
{\colorr $\mathcal{K}$のcptnessはTheorem \ref{main}にもTheorem \ref{main2}にも使う.}
\end{discuss}
a prior density function $\pi_k:\Theta_k\to [0,\infty)$ be a continuous function satisfying $0<\inf_{\theta^{\ast}\in\mathcal{K}}\pi_k(\theta^{\ast}_k)$ and $\sup_{\theta_k\in\Theta_k}\pi_k(\theta_k)<\infty$ for $1\leq k\leq K$.
Let $\theta^{\star}\in\Theta$ and $w=(w_1,\cdots,w_K)\in {\bf W}$. 
An {\bf adaptive Bayes type estimator} $\tilde{\theta}_T=(\tilde{\theta}^1_T,\cdots \tilde{\theta}^K_T)$ is a random variable satisfying
\begin{equation*}
\tilde{\theta}^k_T=\tilde{\theta}^k_T(w)={\rm argmin}_{z_k}\int_{\Theta_k}w_k((a^k_T)^{-1}(z_k-\theta_k))\exp(H_T(\underline{\tilde{\theta}}_{k-1},\theta_k,\overline{\theta^{\star}}_{k+1}))\pi_k(\theta_k)d\theta_k \quad (1\leq k\leq K),
\end{equation*} 
where $\underline{\tilde{\theta}}_{k-1}=\underline{\tilde{\theta}_T}_{k-1}$.
\begin{discuss}
{\colorr $a^k_T$を$\theta^{\ast}$に依るとここでベイズ推定量が真の値を用いて定義されるからまずい.}
\end{discuss}
If the loss function $w$ is any function in Example \ref{loss-function-ex}, 
it is easy to see that an adaptive Bayes type estimator exists.
\begin{discuss}
{\colorr $w$の中身に$(a^k_T)^{-1}$を掛けるから, 漸近的にBayes推定量の存在を言うのは厳しいか.}
\end{discuss}

Adaptive estimation is an estimation method to reduce the calculation cost by calculating estimators separately for each $\Theta_k$.
In certain statistical models like models of ergodic diffusion processes or models of ergodic jump diffusion processes, 
we can adaptively calculate (quasi-)maximum likelihood estimators and Bayes (type) estimators
with the same asymptotic variance as that of {\it simultaneous} estimation.
See Uchida and Yoshida \cite{uch-yos01}, Yoshida \cite{yos05}, Ogihara and Yoshida \cite{ogi-yos01}.
The usual {\it simultaneous} estimation is contained in our setting as the case $K=1$.

\begin{discuss}
{\colorg
$a^1_T$の中でレートを変えることはできるのか? できるなら$K=1$の枠組みで扱えるか?
}
{\colorr
PLDを示すときやベイズ推定量の一致性・漸近正規性には, $H$の微分可能性が必要なかったが, 
この論文では, $\log Z(u)-\Gamma[\hat{u}_T,\hat{u}_T]/2+\Gamma[u_k-\hat{u}^k_T,u_k-\hat{u}^k_T]\to^p 0$
を示すのに本質的に微分可能性が必要か. 
}
\end{discuss}

Let $\hat{u}^k_T=(a^k_T)^{-1}(\hat{\theta}^k_T-\theta^{\ast}_k)$, $\tilde{u}^k_T=\tilde{u}^k_T(w)=(a^k_T)^{-1}(\tilde{\theta}^k_T(w)-\theta^{\ast}_k)$,
$\hat{u}_T=(\hat{u}^1_T,\cdots, \hat{u}^K_T)$, $\tilde{u}_T=(\tilde{u}^1_T,\cdots, \tilde{u}^K_T)$,
$U^k_T(\theta^{\ast}_k)=\{u_k\in\mathbb{R}^{d_k};\theta^{\ast}_k+a^k_Tu_k\in \Theta_k\}$, $V^k_T(r,\theta^{\ast}_k)=\{u_k\in U^k_T(\theta^{\ast}_k);r\leq |u|\}$
and $\psi^k_T(z)=\int_{U^k_T(\theta^{\ast})} w_k(u_k-z)Z^k_T(u_k;\underline{\tilde{\theta}}_{k-1},\theta^{\ast}_k,\overline{\theta^{\star}}_{k+1})\pi_k(\theta^{\ast}_k+a^k_Tu_k)du_k$,
then $\tilde{u}^k_T$ minimizes the function $\psi^k_T(z)$. 
\begin{discuss}
{\colorr Lem \ref{psi-large-est-lemma}でもminimizeする事実を使うし, 重要な性質なので残しておく.}
\end{discuss}
Moreover, we define a random bilinear form $\Gamma^k(\overline{\theta}_{k+1},\theta^{\ast}):\mathbb{R}^{d_k}\times \mathbb{R}^{d_k}\times \mathcal{X}\to \mathbb{R}$ and
\begin{equation*}
\Gamma^k_T(\theta_k,\overline{\theta}_{k+1})[u_k,u_k]=-\partial_{\theta_k}^2H_T(\underline{\tilde{\theta}}_{k-1},\overline{\theta}_k)[a^k_Tu_k,a^k_Tu_k].
\end{equation*}
To avoid redundancy, we denote by $T_0$ a positive constant varying from line to line.  \\

We will state assumptions to obtain asymptotic properties of the Bayes type estimator $\tilde{\theta}_T$.
We will consider uniform estimate in $\theta^{\ast}\in \mathcal{K}$ to obtain convergence results of the estimator uniformly in $\theta^{\ast}\in\mathcal{K}$.
The first one is so-called polynomial type large deviation inequalities. Let $p>0$.
\begin{description}
\item{[$A1\mathchar`- p$]} For $1\leq k\leq K$, there exists $L>1$ and $D_k>d_k+p$ such that
\begin{equation*}
\sup_{\theta^{\ast}\in \mathcal{K},T\geq T_0}P_{\theta^{\ast}}\bigg[\sup_{u_k\in V^k_T(r,\theta^{\ast}_k)}Z^k_T(u_k;\underline{\tilde{\theta}}_{k-1},\theta^{\ast}_k,\overline{\theta^{\star}}_{k+1})\geq \frac{C_1}{r^{D_k}}\bigg]\leq \frac{C_2}{r^L} \quad (r>0).
\end{equation*}
\end{description}
This condition enables us to estimate tail probability of the estimation error
and plays an important role in the proof of asymptotic properties of the Bayes type estimator.
Sufficient conditions of $[A1\mathchar`-p]$ can be found in Yoshida \cite{yos05}.
He proved these inequalities by assuming some conditions on $H_T$ and its derivatives.

Moreover, we assume the following conditions $[A2]$-$[A4]$ with respect to $\Gamma^k$ and the derivatives of $H_T$.
\begin{description}
\item{[$A2$]} For $1\leq k\leq K$, $\{(b^k_T)^{-1}\sup_{\overline{\theta}_k\in \overline{\Theta}_k}|\partial_{\theta_k}^3H_T(\underline{\tilde{\theta}}_{k-1},\theta_k,\overline{\theta}_{k+1})|\}_{T\geq T_0}$ is $P_{\theta^{\ast}}$-tight uniformly $\theta^{\ast}\in \mathcal{K}$,
that is, for any $\epsilon>0$, there exists $M>0$ such that
\begin{equation*}
\sup_{\theta^{\ast}\in \mathcal{K}}\sup_{T\geq T_0}P_{\theta^{\ast}}\bigg[(b^k_T)^{-1}\sup_{\overline{\theta}_k\in \overline{\Theta}_k}|\partial_{\theta_k}^3H_T(\underline{\tilde{\theta}}_{k-1},\theta_k,\overline{\theta}_{k+1})|>M\bigg]<\epsilon.
\end{equation*}
Moreover, $\Gamma^k(\overline{\theta^{\star}}_{k+1},\theta^{\ast})$ is tight uniformly in $\theta^{\ast}\in\mathcal{K}$ and 
$\Gamma^k_T(\theta^{\ast}_k,\overline{\theta^{\star}}_{k+1})\to \Gamma^k(\overline{\theta^{\star}}_{k+1},\theta^{\ast})$
as $T\to\infty$ in $P_{\theta^{\ast}}$-probability uniformly in $\theta^{\ast}\in\mathcal{K}$.
\begin{discuss}
{\colorr $\Gamma^k$は一つの確率変数が$\theta^{\ast}\in\mathcal{K}$に対してuniformlyにtightであるといっている.}
\end{discuss}
\item{[$A3$]} For $1\leq k<l\leq K$, $\sup_{\overline{\theta}_k\in\overline{\Theta}_k}|\partial_{\theta_k}\partial_{\theta_l}H_T(\underline{\tilde{\theta}}_{k-1},\overline{\theta}_k)a^k_T|\to 0$ in $P_{\theta^{\ast}}$-probability uniformly in $\theta^{\ast}\in \mathcal{K}$.

\item{[$A4$]}
At least one of the following two conditions holds true.
\begin{enumerate}
\item $\mathcal{K}=\{\theta^{\ast}\}$ for some $\theta^{\ast}\in\Theta$ and $\Gamma^k(\overline{\theta^{\star}}_{k+1},\theta^{\ast})>0$, $P_{\theta^{\ast}}$-a.s. for $1\leq k\leq K$.
\item $\mathcal{X}$ is a Polish space, $\mathcal{A}$ is the sets of all Borel subsets of $\mathcal{X}$, 
$\Gamma^k(\overline{\theta^{\star}}_{k+1},\theta^{\ast})(x)$ is continuous with respect to $x\in \mathcal{X}$ for $\theta^{\ast}\in\mathcal{K}$,
and $\{P_{\theta^{\ast}}\}_{\theta^{\ast}\in\mathcal{K}}$ is continuous with respect to weak topology.
Moreover, for any $\epsilon,\delta>0$, there exist $\eta>0$ such that
$\sup_{\theta^{\ast}\in\mathcal{K}}P_{\theta^{\ast}}[\lambda_{\min}(\Gamma^k(\overline{\theta^{\star}}_{k+1},\theta^{\ast}))\leq \eta]< \epsilon$
and 
\begin{equation*}
\sup_{\theta^{\ast}\in\mathcal{K}}\sup_{\theta,\theta'\in\mathcal{K};|\theta-\theta'|\leq \eta}P_{\theta^{\ast}}[|\Gamma^k(\overline{\theta^{\star}}_{k+1},\theta')-\Gamma^k(\overline{\theta^{\star}}_{k+1},\theta)|>\delta]<\epsilon
\end{equation*}
for $1\leq k\leq K$.
\end{enumerate}
\end{description}
Condition $[A4]$ is a condition on $\Gamma^k$. If $\mathcal{K}$ consists of one point, that is, we do not consider uniform convergence,
then we need only nondegeneracy of $\Gamma^k$. 
However, we need more conditions on $\Gamma^k$ to obtain uniform convergence.
Conditions $[A2]$ and $[A4] \ 1.$ are usually obtained when we prove polynomial type large deviation inequalities by the scheme of Yoshida \cite{yos05}.
We can easily verify Condition $[A3]$ for statistical models of ergodic diffusion processes and ergodic jump diffusion processes. See Yoshida \cite{yos05} and Ogihara and Yoshida \cite{ogi-yos01}.
Moreover, this condition is nothing if $K=1$.
Since we obtain an explicit form of $\Gamma^k$ for several statistical models, Condition $[A4] \ 2.$ is often not difficult to verify.  

We also assume the following condition for the loss function. 
\begin{description}
\item{[$A5$]} For any $M>0$, there exists $M'>0$ such that
\begin{equation}\label{far-wk-est}
\sup\{w_k(u_k);|u_k|\leq M\}-\inf\{w_k(u_k);|u_k|\geq M'\}\leq 0
\end{equation}
for $1\leq k\leq K$.
\end{description}
This type of condition is necessary 
to obtain asymptotic properties of Bayes estimator in models of i.i.d. observation. See Theorem 5.2. in Chapter I of Ibragimov and Has'minskii \cite{ibr-has03}.

Finally, we assume some conditions on the M-estimator and the Bayes type estimator.
\begin{description}
\item{[$A6$]} The sequence $\{\hat{u}_T\}_{T\geq T_0}$ is $P_{\theta^{\ast}}$-tight uniformly in $\theta^{\ast}\in\mathcal{K}$.
\item{[$A7$]} An adaptive Bayes type estimator $\tilde{\theta}_T$ exists $P_{\theta^{\ast}}$-a.s. for any $\theta^{\ast}\in\mathcal{K}$ and $T\geq T_0$.
\end{description}

\begin{theorem}\label{main}
Let $p>0$ and $w\in{\bf W}_p$. Assume $[A1\mathchar`-p]$ and $[A2]$-$[A7]$. 
Then $\tilde{u}_T(w)-\hat{u}_T\to 0$ as $T\to\infty$ in $P_{\theta^{\ast}}$-probability uniformly in $\theta^{\ast}\in \mathcal{K}$.
\end{theorem}
This theorem implies that if we specify the asymptotic distribution of the estimation error of M-estimator, 
the estimation error of Bayes type estimator converges to the same limit.\\

We consider convergence of moments in the following. 
Let $\eta\in (0,1)$.
\begin{description}
\item{[$C1\mathchar`-\eta$]} There exists $r_0>0$ such that the loss function $w_k$ satisfies
\begin{equation*}
\inf_{r\geq r_0}\inf_{|u_k|\leq r^{\eta},|z|\geq r}(w_k(u_k-z)-w_k(u_k))\geq 0
\end{equation*} 
for $1\leq k\leq K$. 
\item{[$C2\mathchar`-q,\eta$]} There exist $L_1>q/\eta$ and $\{C_R\}_{R>0}\subset (0,\infty)$ such that 
\begin{equation*} 
\sup_{\theta^{\ast}\in\mathcal{K},T\geq T_0}P_{\theta^{\ast}}\bigg[\inf_{|u_k|\leq R}\log(Z^k_T(u_k;\underline{\tilde{\theta}}_{k-1},\theta^{\ast}_k,\overline{\theta^{\star}}_{k+1}))\leq -r\bigg]\leq \frac{C_R}{r^{L_1}}
\end{equation*}
for any $R>0$ and $r>0$.
\item{[$C3\mathchar`-q,\eta$]} There exist $L_2>q/\eta$ and $C>0$ such that
\begin{equation*}
\sup_{\theta^{\ast}\in \mathcal{K},T\geq T_0}P_{\theta^{\ast}}\bigg[\sup_{u_k\in V^k_T(r,\theta^{\ast}_k)}Z^k_T(u_k;\underline{\tilde{\theta}}_{k-1},\theta^{\ast}_k,\overline{\theta^{\star}}_{k+1})\geq e^{-r}\bigg]\leq \frac{C}{r^{L_2}}
\end{equation*}
for any $r>0$.
\end{description}
Condition $[C3\mathchar`-q,\eta]$ is another version of polynomial type large deviation inequalities, 
and also proved by the scheme of Yoshida \cite{yos05}.
Condition $[C2\mathchar`-q,\eta]$ is usually obtained when we use the scheme of Yoshida \cite{yos05}.  
It is easy to see that $[C1\mathchar`-\eta]$ implies $[A5]$ for any $\eta\in(0,1)$.
\begin{discuss}
{\colorr $w_k(u^1_k)>w_k(u^2_k)$, $|u^1_k|\leq r^{\eta},|u^2_k|\geq r+r^{\eta}$とすると, $z=u^1_k-u^2_k,u_k=u^1_k$とおくと, $[C1\mathchar`-\eta]$に矛盾.}
\end{discuss}

We denote by $E_{\theta^{\ast}}$ the expectation with respect to $P_{\theta^{\ast}}$.
\begin{theorem}\label{main2}
Let $\eta\in (0,1)$, $q>1$ and $w\in{\bf W}$. Assume $[A7]$,$[C1\mathchar`-\eta],[C2\mathchar`-q,\eta],[C3\mathchar`-q,\eta]$ and $\inf_{\theta_k\in\Theta_k}\pi_k(\theta_k)>0$.
Then there exists $T_0>0$ such that $\sup_{\theta^{\ast}\in\mathcal{K},T\geq T_0}E_{\theta^{\ast}}[|\tilde{u}_T(w)|^q]<\infty$. 
\end{theorem}

Let $\hat{u}(\theta^{\ast})$ be a random variable on another statistical experiment $(\tilde{\mathcal{X}},\tilde{\mathcal{A}},\{\tilde{P}\}_{\theta\in\Theta})$,
and $\tilde{E}_{\theta}$ represent the expectation with respect to $\tilde{P}_{\theta}$.
\begin{corollary}\label{main-cor}
Assume that $\mathcal{L}(\hat{u}_T|P_{\theta^{\ast}})\to \mathcal{L}(\hat{u}(\theta^{\ast})|\tilde{P}_{\theta^{\ast}})$ as $T\to\infty$ uniformly in $\theta^{\ast}\in \mathcal{K}$. 
\begin{enumerate}
\item Let $p>0$ and $w\in{\bf W}_p$. Assume $[A1\mathchar`-p]$, $[A2]$-$[A7]$. 
Then $\mathcal{L}(\tilde{u}_T(w)|P_{\theta^{\ast}})\to \mathcal{L}(\hat{u}(\theta^{\ast})|\tilde{P}_{\theta^{\ast}})$ as $T\to\infty$ uniformly in $\theta^{\ast}\in \mathcal{K}$. 
\item Let $\eta\in (0,1)$, $q>1$ and $w\in{\bf W}$. Assume $[A2]$-$[A4]$, $[A6]$, $[A7]$, $[C1\mathchar`-\eta],[C2\mathchar`-q,\eta],[C3\mathchar`-q,\eta]$ and $\inf_{\theta_k\in\Theta_k}\pi_k(\theta_k)>0$ for $1\leq k\leq K$.
Then $E_{\theta^{\ast}}[f(\tilde{u}_T(w))]\to \tilde{E}_{\theta^{\ast}}[f(\hat{u})]$ uniformly in $\theta^{\ast} \in \mathcal{K}$ 
for any continuous function $f$ satisfying $\limsup_{|u|\to\infty}|f(u)||u|^{-q}<\infty$.
\end{enumerate}
\end{corollary}
\begin{discuss}
{\colorr $\hat{u}$のモーメントの有限性は, 分布収束なら空間を代えてa.s.収束にできることと, Fatou's lemmaを使う.}
{\colorg 分布収束ならa.s.収束を示すときの条件は満たされているのか?}
\end{discuss}

\section{An application to ergodic diffusion processes}\label{ergodic-diffusion-section}

We will see an application of our results to statistical models of ergodic diffusion processes.
We consider the setting of Yoshida \cite{yos05}.

Let $(\Omega,\mathcal{F},P)$ be a probability space and ${\bf F}=\{\mathcal{F}_t\}_{t\geq 0}$ be a filtration. 
We consider a $m$-dimensional ${\bf F}$-adapted process $X=\{X_t\}_{t\geq 0}$ satisfying the following stochastic differential equation:
\begin{equation*}
dX_t=a(X_t,\theta^{\ast}_2)dt+b(X_t,\theta^{\ast}_1)dW_t,\quad t\geq 0,
\end{equation*}
where $\{W_t\}_{t\geq 0}$ is an $r$-dimensional ${\bf F}$-standard Wiener process, 
$a:\mathbb{R}^m\times \Theta_2\to\mathbb{R}^m$ and $b:\mathbb{R}^m\times \Theta_1\to\mathbb{R}^m\otimes \mathbb{R}^r$ are Borel functions.
$\theta^{\ast}_1\in\Theta_1$ and $\theta^{\ast}_2\in \Theta_2$ are unknown parameters.
We assume that $\Theta_1\subset \mathbb{R}^{d_1}$ and $\Theta_2\subset \mathbb{R}^{d_2}$ are bounded convex open sets satisfying Sobolev's inequalities, that is,
for $i=1,2$ and any $p>d_i$, there exists $C>0$ such that
\begin{equation*}
\sup_{x\in\Theta_i}|f(x)|\leq C\sum_{k=0,1}\parallel \partial_x^kf(x)\parallel_p, \quad (f\in C^1(\Theta_i)).
\end{equation*}
It is the case if $\Theta_1$ and $\Theta_2$ have Lipschitz boundaries. See Adams \cite{ada}, Adams and Fournier \cite{ada-fou}. 
\begin{discuss}
{\colorg convexなら自動か?}
\end{discuss}
We also assume that $\Theta$ satisfies Sobolev's inequalities. The distribution of $X_0$ may depend on $\theta^{\ast}=(\theta^{\ast}_1,\theta^{\ast}_2)$.

Let $B(x,\theta_1)=bb^{\top}(x,\theta_1)$. We assume the following conditions.
\begin{description}
\item{[$D1$]} 
\begin{enumerate}
\item $E[|X_0|^q]<\infty$ for any $q>0$.
\item $B(x,\theta_1)$ is elliptic uniformly in $(x,\theta_1)$.
\item The derivatives $\partial^i_{\theta_2}a$ and $\partial^j_x\partial^i_{\theta_1}b$ exist and continuous,
and there exists constant $C>0$ such that
\begin{equation*} 
\sup_{\theta_2\in\Theta_2}|\partial^i_{\theta_2}a(x,\theta_2)|\leq  C(1+|x|)^C, \quad
\sup_{\theta_1\in\Theta_1}|\partial^j_x\partial^i_{\theta_1}b(x,\theta_1)|\leq C(1+|x|)^C 
\end{equation*}
for any $x\in\mathbb{R}^m$, $0\leq i\leq 4$ and $0\leq j\leq 2$.
Moreover, $a$ and $b$ can be extended to continuous functions on $\mathbb{R}^m\times {\rm clos}(\Theta_2)$ and $\mathbb{R}^m\times {\rm clos}(\Theta_1)$, respectively.
\item There exists a constant $C>0$ such that
\begin{equation*}
\sup_{\theta_2\in\Theta_2}|a(x_1,\theta_2)-a(x_2,\theta_2)|+\sup_{\theta_1\in\Theta_1}|b(x_1,\theta_1)-b(x_2,\theta_1)|\leq C|x_1-x_2|
\end{equation*}
for $x_1,x_2\in\mathbb{R}^m$.
\end{enumerate}
\item {[$D2$]} There exists a positive constant $c$ such that
\begin{equation*}
\sup_{t\geq 0}\sup_{A\in\sigma(X_s;s\leq t),B\in\sigma(X_s;s\geq t+h)}|P[A\cap B]-P[A]P[B]|\leq c^{-1}\exp(-ch) \quad (h>0).
\end{equation*}
\end{description}
Condition $[D2]$ implies ergodicity of $X$: there exists an invariant measure $\nu$ for $X_t$ such that
\begin{equation*}
\frac{1}{T}\int^T_0g(X_t)dt\to \int_{\mathbb{R}^m}g(x)\nu(dx)
\end{equation*}
as $T\to\infty$ for any bounded measurable function $g$. 
\begin{discuss}
{\colorg
For the mixing properties $[D2]$ of diffusion processes, we refer the reader to Kusuoka and Yoshida \cite{kus-yos}}. {\colorg Kusuoka Yoshidaで十分か?} 
\end{discuss}

We consider estimation of the parameter $\theta^{\ast}=(\theta^{\ast}_1,\theta^{\ast}_2)$ by discrete observations $\{X_{ih_n}\}_{i=0}^n$ of $X$,
where $h_n$ is a positive number satisfying $h_n\to 0$, $nh_n\to \infty$ and $nh^2_n\to 0$ as $n\to \infty$.
Moreover, we assume that there exists a constant $\epsilon_0>0$ such that $nh\geq n^{\epsilon_0}$ for sufficiently large $n$.

Yoshida \cite{yos05} considered a quasi-likelihood function $H_n(\theta)$ defined by
\begin{eqnarray}
H_n(\theta)=-\frac{1}{2}\sum_{i=1}^n\bigg\{\frac{B(X_{(i-1)h_n},\theta_1)^{-1}}{h}[(X_{ih_n}-X_{(i-1)h_n}-ha(X_{(i-1)h_n},\theta_2))^{\otimes 2}]+\log \det B(X_{(i-1)h_n},\theta_1)\bigg\}.
\end{eqnarray}

The quasi-maximum likelihood estimator $\hat{\theta}_n$ is defined as a random variable satisfying $\hat{\theta}_n=(\hat{\theta}^1_n,\hat{\theta}^2_n)={\rm argmax}_{\theta\in{\rm clos}(\Theta)}H_n(\theta)$.
Let $\theta^{\star}_2\in \Theta_2$, $w=(w_1,w_2)\in{\bf W}$, a prior density function $\pi=(\pi_1,\pi_2):\Theta\to (0,\infty)$ be continuous and bounded. 
Then the adaptive Bayes type estimator $\tilde{\theta}_n=(\tilde{\theta}^1_n,\tilde{\theta}^2_n)$ is an random variable satisfying
\begin{eqnarray}
\tilde{\theta}^1_n&=&{\rm argmin}_{z_1}\int_{\Theta_1} w_1(\sqrt{n}(z_1-\theta_1))\exp(H_n(\theta_1,\theta^{\star}_2))\pi_1(\theta_1)d\theta_1, \nonumber \\
\tilde{\theta}^2_n&=&{\rm argmin}_{z_2}\int_{\Theta_2} w_2(\sqrt{nh_n}(z_2-\theta_2))\exp(H_n(\tilde{\theta}^1_n,\theta_2))\pi_2(\theta_2)d\theta_2. \nonumber 
\end{eqnarray}

Let 
\begin{eqnarray}
Y^1(\theta_1)&=&-\frac{1}{2}\int_{\mathbb{R}^m}\bigg\{{\rm tr}\left(B(x,\theta_1)^{-1}B(x,\theta^{\ast}_1)-I_m\right)+\log\frac{\det B(x,\theta_1)}{\det B(x,\theta^{\ast}_1)}\bigg\}\nu(dx), \nonumber \\
Y^2(\theta_2)&=&-\frac{1}{2}\int_{\mathbb{R}^m}B(x,\theta^{\ast}_1)^{-1}[(a(x,\theta_2)-a(x,\theta^{\ast}_2))^{\otimes 2}]\nu(dx). \nonumber \\
\Gamma^1&=&\frac{1}{2}\int_{\mathbb{R}^m} {\rm tr}\left\{B^{-1}(\partial_{\theta_1}B)B^{-1}(\partial_{\theta_1}B)(x,\theta^{\ast}_1)\right)\nu(dx), \nonumber \\
\Gamma^2&=&\int_{\mathbb{R}^m} (\partial_{\theta_2}a(x,\theta^{\ast}_2))^{\top}B(x,\theta^{\ast}_1)^{-1}\partial_{\theta_2}a(x,\theta^{\ast}_2)\nu(dx), \nonumber
\end{eqnarray}
where $I_m$ represents the unit matrix of size $m$.
We assume some more conditions.
\begin{description}
\item{[$D3$]} There exists a positive constant $\chi_1$ such that $Y^1(\theta_1)\leq -\chi_1|\theta_1-\theta^{\ast}_1|^2$ for any $\theta_1\in\Theta_1$. 
\item{[$D4$]} There exists a positive constant $\chi_2$ such that $Y^2(\theta_2)\leq -\chi_2|\theta_2-\theta^{\ast}_2|^2$ for any $\theta_2\in\Theta_2$. 
\item{[$D5$]} An adaptive Bayes type estimator $\tilde{\theta}_n$ exists a.s. for sufficiently large $n$ and 
there exist constants $r_0>0$ and $\eta \in (0,1)$ such that the loss function $w_k$ satisfies
\begin{equation*}
\inf_{r\geq r_0}\inf_{|u_k|\leq r^{\eta},|z|\geq r}(w_k(u_k-z)-w_k(u_k))\geq 0
\end{equation*}
for $1\leq k\leq 2$.
\end{description}

Let $(\zeta_1,\zeta_2)$ is a zero-mean normal random variable with variance ${\rm diag}((\Gamma^1)^{-1},(\Gamma^2)^{-1})$.
\begin{theorem}\label{ergodic-diffusion-theorem}
Assume $[D1]$-$[D5]$ and that $\inf_{\theta_1\in\Theta_1}\pi_1(\theta_1)\wedge \inf_{\theta_2\in\Theta_2}\pi_2(\theta_2)>0$. Then 
\begin{equation*}
(\sqrt{n}(\tilde{\theta}^1_n-\theta^{\ast}_1),\sqrt{nh_n}(\tilde{\theta}^2_n-\theta^{\ast}_2))\to^d (\zeta_1,\zeta_2)
\end{equation*}
as $n\to \infty$. Moreover, 
\begin{equation*}
E[f(\sqrt{n}(\tilde{\theta}^1_n-\theta^{\ast}_1),\sqrt{nh_n}(\tilde{\theta}^2_n-\theta^{\ast}_2))]\to E[f(\zeta_1,\zeta_2)]
\end{equation*}
as $n\to \infty$ for any continuous function $f$ of at most polynomial growth.
\end{theorem}

\section{Proofs}\label{proofs-section}

We will prove Theorems \ref{main}, \ref{main2} and \ref{ergodic-diffusion-theorem}. We apply the idea of the proof of Theorem 8.2. in Ibragimov and Has'minskii \cite{ibr-has03}.
First, we prepare some lemmas.
\begin{lemma}\label{psi-large-est-lemma}
Let $p>0$, $w\in{\bf W}_p$ and $1\leq k\leq K$. Assume $[A1\mathchar`-p]$, $[A2]$, $[A3]$ and $[A5]$-$[A7]$. 
Then $\{\tilde{u}^k_T(w)\}_{T\geq T_0}$ is $P_{\theta^{\ast}}$-tight uniformly in $\theta^{\ast}\in\mathcal{K}$. 
\end{lemma}
\begin{proof}
We will prove by induction on $k$. Assume $\{\tilde{u}^l_T(w)\}_{T\geq T_0}$ is $P_{\theta^{\ast}}$-tight uniformly in $\theta^{\ast}\in\mathcal{K}$ for $1\leq l\leq k-1$. 
Fix $\epsilon>0$. Taylor's formula yields 
\begin{eqnarray}\label{logZ-exp}
\log(Z^k_T(u_k;\underline{\tilde{\theta}}_{k-1},\theta^{\ast}_k,\overline{\theta^{\star}}_{k+1}))
&=&\partial_{\theta_k}H_T(\underline{\tilde{\theta}}_{k-1},\theta^{\ast}_k,\overline{\theta^{\star}}_{k+1})[a^k_Tu_k]
+\partial^2_{\theta_k}H_T(\underline{\tilde{\theta}}_{k-1},\theta^{\ast}_k,\overline{\theta^{\star}}_{k+1})[(a^k_Tu_k)^{\otimes 2}]/2 \nonumber \\
&&\int^1_0\frac{(1-t)^2}{2}\partial^3_{\theta_k}H_T(\underline{\tilde{\theta}}_{k-1},\theta^{\ast}_k+ta^k_Tu_k,\overline{\theta^{\star}}_{k+1})[(a^k_Tu_k)^{\otimes 3}]dt, 
\end{eqnarray}
if $\{\theta^{\ast}_k+ta^k_Tu_k\}_{0\leq t\leq 1}\subset \Theta_k$.

Moreover, we have
\begin{eqnarray}\label{partialH-est}
&&\partial_{\theta_k}H_T(\underline{\tilde{\theta}}_{k-1},\theta^{\ast}_k,\overline{\theta^{\star}}_{k+1})[a^k_Tu_k] \nonumber \\
&=&\partial_{\theta_k}H_T(\underline{\tilde{\theta}}_{k-1},\theta^{\ast}_k,\overline{\theta^{\star}}_{k+1})[a^k_Tu_k] -\partial_{\theta_k}H_T(\hat{\theta})[a^k_Tu_k] \nonumber \\
&=&-\int^1_0\partial^2_{\theta_k}H_T(\underline{\tilde{\theta}}_{k-1},\theta^{\ast}_k+t(\hat{\theta}_k-\theta^{\ast}_k),\overline{\theta^{\star}}_{k+1})[a^k_Tu_k,\hat{\theta}^k_T-\theta^{\ast}_k]dt \nonumber \\
&&+\int^1_0\partial_{\overline{\theta}_{k+1}}\partial_{\theta_k}H_T(\underline{\tilde{\theta}}_{k-1},\hat{\theta}^k_T,\overline{\hat{\theta}}_{k+1}+t(\overline{\theta^{\star}}_{k+1}-\overline{\hat{\theta}}_{k+1}))[a^k_Tu_k,\overline{\theta^{\star}}_{k+1}-\overline{\hat{\theta}}_{k+1}]dt, \nonumber \\
&&+\sum_{l;l<k}\int^1_0\partial_{\theta_l}\partial_{\theta_k}H_T(\underline{\tilde{\theta}}_{l-1},\hat{\theta}^l_T+t(\tilde{\theta}^l_T-\hat{\theta}^l_T),\overline{\hat{\theta}}_{l+1})[a^k_Tu_k,\tilde{\theta}^l_T-\hat{\theta}^l_T]dt,
\end{eqnarray}
if $\hat{\theta}\in\Theta$ and $\{t\theta^{\ast}_k+(1-t)\hat{\theta}^k_T\}_{0\leq t\leq 1}\subset \Theta_k$.

Hence for any $M>0$, there exists $R>0$ such that $\sup_{T\geq T_0}P[\sup_{|u_k|\leq M}|\log(Z^k_T(u_k;\underline{\tilde{\theta}}_{k-1},\theta^{\ast}_k,\overline{\theta^{\star}}_{k+1}))|>R]<\epsilon/2$, 
by $[A2]$, $[A3]$, $[A6]$ and the induction assumption.

Moreover, $w_k\in{\bf W}_{p,k}$ implies that there exist positive constants $\{\delta_j\}_{j=1}^3$ such that
\begin{equation}\label{wk-large-est}
\inf_{|z|>\delta_3}\int_{|u_k|\leq \delta_1}(w_k(u_k-z)-w_k(u_k))du_k>\delta_2.
\end{equation}
\begin{discuss}
{\colorr
${\bf W}_{p,k}$の条件3から, ある$M>0$と$c>0$があって, $w_k(u_k)>c$ if $|u_k|\geq M$.
一方, 条件1からある$\delta_1$があって, $w_k(u_k)<c/2$ on $|u_k|\leq \delta_1$.
よって$\delta_3=M+\delta_1$, $\delta_2=c/2$ととればよい.
}
\end{discuss}

Furthermore, by the virtue of $[A1\mathchar`-p]$, there exists $M>\delta_1\vee \delta_3$ such that 
\begin{equation*}
\sup_{\theta^{\ast}\in\mathcal{K},T\geq T_0}P_{\theta^{\ast}}\bigg[\int_{\{|u_k|> M\}\cap U^k_T(\theta^{\ast})}w_k(u_k)Z^k_T(u_k;\underline{\tilde{\theta}}_{k-1},\theta^{\ast}_k,\overline{\theta^{\star}}_{k+1})\pi_k(\theta^{\ast}_k+a^k_Tu_k)du_k>Q]<\epsilon/2,
\end{equation*}
where $Q=\delta_2e^{-R}\inf_{\theta_k^{\ast}\in\mathcal{K}}\pi_k(\theta^{\ast}_k)/2$.
\begin{discuss}
{\colorr $\sup \pi$の評価ここで使う.}
\end{discuss}

On the other hand, there exists $M'>M$ such that (\ref{far-wk-est}) holds true.
Hence we obtain
\begin{eqnarray}
&&\sup_{\theta^{\ast}\in\mathcal{K},T\geq T_0}P_{\theta^{\ast}}[|\tilde{u}^k_T|\geq M'+M] \nonumber \\
&\leq &\sup_{\theta^{\ast}\in\mathcal{K},T\geq T_0}P_{\theta^{\ast}}\bigg[\inf_{|z|>M'+M}\int_{\{|u_k|\leq M\}\cap U^k_T(\theta^{\ast})}w_k(u_k-z)Z^k_T(u_k;\underline{\tilde{\theta}}_{k-1},\theta^{\ast}_k,\overline{\theta^{\star}}_{k+1})\pi_k(\theta^{\ast}_k+a^k_Tu_k)du_k \nonumber \\
&&\quad \quad \quad \leq \int_{\{|u_k|\leq M\}\cap U^k_T(\theta^{\ast})}w_k(u_k)Z^k_T(u_k;\underline{\tilde{\theta}}_{k-1},\theta^{\ast}_k,\overline{\theta^{\star}}_{k+1})\pi_k(\theta^{\ast}_k+a^k_Tu_k)du_k+Q\bigg]+\frac{\epsilon}{2} \nonumber \\
&\leq & \sup_{\theta^{\ast}\in\mathcal{K},T\geq T_0}P_{\theta^{\ast}}\bigg[\bigg(\frac{1}{2}\inf_{\theta^{\ast}\in\mathcal{K}}\pi_k(\theta_k^{\ast})\bigg)e^{-R}\inf_{|z|>M'+M}\int_{|u_k|\leq M}(w_k(u_k-z)-w_k(u_k))du_k\leq Q\bigg]+\epsilon=\epsilon. \nonumber 
\end{eqnarray}
\begin{discuss}
{\colorr
$T_0$を十分大きくとれば, $\theta^{\ast}_k$に依れば, $|\pi_k(\theta^{\ast}_k+a^k_Tu_k)-\pi_k(\theta^{\ast}_k)|<\inf_{\theta^{\ast}\in\mathcal{K}}\pi_k(\theta^{\ast}_k)/2$ととれる.
}
\end{discuss}
\end{proof}

\begin{lemma}\label{logZ-lim-lemma}
Let $p>0$, $w\in{\bf W}_p$ and $1\leq k\leq K$. Assume $[A1\mathchar`-p]$,$[A2]$,$[A3]$ and $[A5]$-$[A7]$. Then
\begin{equation*}
\sup_{|u_k|\leq R}\bigg|\log Z^k_T(u_k;\underline{\tilde{\theta}(w)}_{k-1},\theta^{\ast}_k,\overline{\theta^{\star}}_{k+1})
-\frac{1}{2}\Gamma^k(\overline{\theta^{\star}}_{k+1},\theta^{\ast})[\hat{u}^k_T,\hat{u}^k_T]
+\frac{1}{2}\Gamma^k(\overline{\theta^{\star}}_{k+1},\theta^{\ast})[u_k-\hat{u}^k_T,u_k-\hat{u}^k_T]\bigg|\to 0
\end{equation*}
as $T\to\infty$ in $P_{\theta^{\ast}}$-probability uniformly in $\theta^{\ast}\in\mathcal{K}$ for any $R>0$.
\end{lemma}
\begin{proof}
By $[A2]$ and (\ref{logZ-exp}), we obtain 
\begin{equation*}
\sup_{|u_k|\leq R}\bigg|\log Z^k_T(u_k;\underline{\tilde{\theta}}_{k-1},\theta^{\ast}_k,\overline{\theta^{\star}}_{k+1})
-\partial_{\theta_k}H_T(\underline{\tilde{\theta}}_{k-1},\theta^{\ast}_k,\overline{\theta^{\star}}_{k+1})[a^k_Tu_k]
+\frac{1}{2}\Gamma^k(\overline{\theta^{\star}}_{k+1},\theta^{\ast})[u_k,u_k]\bigg|\to 0
\end{equation*}
as $T\to\infty$ in $P_{\theta^{\ast}}$-probability uniformly in $\theta^{\ast}\in\mathcal{K}$ for any $R>0$. 
\begin{discuss}
{\colorr
$|a_k|^2=|a_k^{\star}a_k|=\lambda_{\max}(a_k^{\star}a_k)\leq C_1b_k^{-1}$.
}
\end{discuss}

Moreover, we obtain
\begin{equation*}
\sup_{|u_k|\leq R}\bigg|\log Z^k_T(u_k;\underline{\tilde{\theta}}_{k-1},\theta^{\ast}_k,\overline{\theta^{\star}}_{k+1})
-\Gamma^k(\overline{\theta^{\star}}_{k+1},\theta^{\ast})[u_k,\hat{u}^k_T]+\frac{1}{2}\Gamma^k(\overline{\theta^{\star}}_{k+1},\theta^{\ast})[u_k,u_k]\bigg|\to 0
\end{equation*}
as $T\to\infty$ in $P_{\theta^{\ast}}$-probability uniformly in $\theta^{\ast}\in\mathcal{K}$ for any $R>0$, 
by (\ref{partialH-est}), $[A2]$, $[A3]$, $[A6]$ and Lemma \ref{psi-large-est-lemma}.
\end{proof}

\begin{discuss}
{\colorr
Ibragimov and Has'minskii Lemma 2.10.2.

Let the probability density $f_{\xi}(x)$ of a random variable $\xi$ in $\mathbb{R}^d$ satisfy $f_{\xi}(x)=f_{\xi}(-x)$ 
and the sets $\{x;f_{\xi}(x)\geq u\}$ are convex for all $u\geq 0$.
Let $l:\mathbb{R}^d\to [0,\infty)$ be a function such that $l(0)=0$, $l(x)=l(-x)$ and the set $\{x;l(x)<c\}$ be convex for any $c>0$, 
$E[l(\xi+y)]<\infty$ for $y\in \mathbb{R}^d$ and $E[l(\xi+y)]$ is continuous with respect to $y$. 
Moreover, assume that $\{x;l(x)<c_1\}\subset \{x;f_{\xi}(x)\geq u_1\}$ and the sets $\{x;f_{\xi}(x)\geq u\}$ are strictly convex and 
$\{x;f_{\xi}(x)= u\}$ is of Lebesgue measure $0$ for some $u_1>0$ and $c_1>0$ and any $u\geq u_1$.
Then $\inf_{|y|>\epsilon}E[l(\xi+y)]> E[l(\xi)]$ for any $\epsilon>0$.
  
}
\end{discuss}

We define 
\begin{equation*}
G_k(z;\theta^{\ast})=\int_{\mathbb{R}^{d_k}}w_k(u_k+z)\exp\bigg(-\frac{1}{2}\Gamma^k(\overline{\theta^{\star}}_{k+1},\theta^{\ast})[u_k,u_k]\bigg)du_k
-\int_{\mathbb{R}^{d_k}}w_k(u_k)\exp\bigg(-\frac{1}{2}\Gamma^k(\overline{\theta^{\star}}_{k+1},\theta^{\ast})[u_k,u_k]\bigg)du_k
\end{equation*}
for $z\in\mathbb{R}^{d_k}$.

\begin{lemma}\label{G-est-lemma}
Assume $[A4]$. Let $0<\delta<R$ and $w\in{\bf W}$. Then for any $\epsilon>0$, there exists $\eta>0$ such that
$\sup_{\theta^{\ast}\in\mathcal{K}}P_{\theta^{\ast}}[\inf_{R\geq |z|\geq \delta}G_k(z;\theta^{\ast})\leq \eta]<\epsilon$.
\end{lemma}
\begin{proof}
We only consider the case $[A4]\ 2.$ is satisfied. The proof of the other case is easier. 
We assume that there exists $\epsilon>0$, $\theta^{\ast}\in \mathcal{K}$ and $\{\theta^n\}_{n\in\mathbb{N}}\subset \mathcal{K}$ such that 
$\theta^n\to\theta^{\ast}$ as $n\to\infty$ and 
\begin{equation}\label{G-contradict-est}
P_{\theta^n}[\inf_{R\geq |z|\geq \delta}G_k(z;\theta^n)\leq 1/n]\geq \epsilon,
\end{equation}
and lead to a contradiction.

By Lemma 2.10.2. in Ibragimov and Has'minskii \cite{ibr-has03}, there exists $n_0\in\mathbb{N}$ such that
\begin{equation}\label{G-contradict-est2}
P_{\theta^{\ast}}[\inf_{R\geq |z|\geq \delta}G_k(z;\theta^{\ast})\leq 1/n_0]<\epsilon/4.
\end{equation}
\begin{discuss}
{\colorr $\Gamma^k[a,a]\leq u$ and $\Gamma^k[b,b]\leq u$なら, $t\in [0,1]$に対し, 
\begin{equation*}
\Gamma^k[ta+(1-t)b,ta+(1-t)b]= t^2\Gamma^k[a,a]+2t(1-t)\Gamma^k[a,b]+(1-t)^2\Gamma^k[b,b]\leq t^2u+2t(1-t)u+(1-t)^2u=u
\end{equation*}
より, $\{x,\exp(-\Gamma^k[x,x]/2)\geq u\}$はconvex.
等号が成り立つのは, $b=sa \ (s\geq 0)$のときだけなので, $a\neq b$なら$\Gamma^k[ta+(1-t)b,ta+(1-t)b]<u$. よって
$\{x,\exp(-\Gamma^k[x,x]/2)\geq u\}$はstrictly convex.

変数変換とLebesgueの収束定理で$G$がcontinuousであることもわかるからOK.
}
\end{discuss}

By $[A4]$, there exists $\eta'>0$ such that 
\begin{equation*}
\sup_{\theta\in\mathcal{K}}P_{\theta}[|\inf_{R\geq |z|\geq \delta}G_k(z;\theta^{\ast})-\inf_{R\geq |z|\geq \delta}G_k(z;\theta')|>1/(2n_0)]<\epsilon/2
\end{equation*}
for any $\theta'$ satisfying $|\theta^{\ast}-\theta'|<\eta'$.
\begin{discuss}
{\colorr
まず, $[A4]$の$\lambda_{\min}$に関する条件を使えば, $G_k(u_k,\theta^{\ast})$の中の$u_k$の積分領域の非有界な部分をカットオフできる.
そうすれば, $[A4]$の$\Gamma^k(\theta^{\star}_{k+1},\theta^{\ast})$の連続性の条件から$G^k(z;\theta^{\ast})-G^k(z;\theta')$を$z$に
関して一様に評価できるので, $\inf$に対しても評価できる.
}
\end{discuss}
Hence there exists $n_1\in\mathbb{N}$ such that 
\begin{equation*}
\sup_{n\geq n_1}P_{\theta^n}[\inf_{R\geq |z|\geq \delta}G_k(z,\theta^{\ast})\leq 1/n_0]\geq \epsilon/2,
\end{equation*}
by (\ref{G-contradict-est}).
\begin{discuss}
{\colorr
途中式：
\begin{eqnarray}
&\geq & P_{\theta^{\ast}}[\inf G_k(z,\theta^n)\leq 1/(2n_0), |\inf G_k(z,\theta^{\ast})-\inf G_k(z,\theta^n)|\leq 1/(2n_0)] \nonumber \\
&\geq & P_{\theta^{\ast}}[\inf G_k(z,\theta^n)\leq 1/(2n_0)]-P_{\theta^{\ast}}[|\inf G_k(z,\theta^{\ast})-\inf G_k(z,\theta^n)|\leq 1/(2n_0)]. \nonumber 
\end{eqnarray}
}
\end{discuss}

Therefore, there exists $\eta''>0$ such that 
\begin{eqnarray}
P_{\theta^{\ast}}[\inf_{R\geq |z|\geq \delta}G_k(z,\theta^{\ast})\leq 1/n_0]
&\geq & P_{\theta^{\ast}}[\inf_{R\geq |z|\geq \delta}G_k(z,\theta^{\ast})1_{\{\Gamma^k(\overline{\theta^{\star}}_{k+1},\theta^{\ast})>\eta''\} }\leq 1/n_0]-\epsilon/4 \nonumber \\
&\geq & \limsup_{n\to\infty}P_{\theta^n}[\inf_{R\geq |z|\geq \delta}G_k(z,\theta^{\ast})1_{\{\Gamma^k(\overline{\theta^{\star}}_{k+1},\theta^{\ast})>\eta''\} }\leq 1/n_0]-\epsilon/4 \nonumber \\
&\geq & \limsup_{n\to\infty}P_{\theta^n}[\inf_{R\geq |z|\geq \delta}G_k(z,\theta^{\ast})\leq 1/n_0]-\epsilon/4 \geq \epsilon/4, \nonumber
\end{eqnarray}
by $[A4]$, which contradicts (\ref{G-contradict-est2}). 
\begin{discuss}
{\colorr
$\inf_{R\geq |z|\geq \delta}G_k(z;\theta^{\ast})1_{\{\Gamma^k(\overline{\theta^{\star}}_{k+1},\theta^{\ast})>\eta''\} }(x)$の$x\in\mathcal{X}$に関する連続性.
任意の$\epsilon>0$に対し, ある$M>0$と$\delta>0$があって, $|x-x'|\leq \delta$ならば
\begin{equation*}
\sup_{R\geq |z|\geq \delta}\int_{|u_k|> M}w_k(u_k+z)\exp(-\Gamma^k(\overline{\theta^{\star}}_{k+1},\theta^{\ast})[u_k,u_k]/2)du_k<\epsilon/2.
\end{equation*}
これより, $|G_k(z,\theta^{\ast},x)-G_k(z,\theta^{\ast},x')|<\epsilon$がいえる. 上と同様に$\inf$も入れる.
}
\end{discuss}
\end{proof}

\noindent
{\bf Proof of Theorem \ref{main}.}

Fix $\epsilon,\delta>0$.
By Lemma \ref{psi-large-est-lemma} and $[A6]$, there exist $R_1>\delta$ and $R_2>0$ such that
\begin{equation*}
\sup_{\theta^{\ast}\in\mathcal{K}}\sup_{T\geq T_0}P_{\theta^{\ast}}[|\tilde{u}^k_T-\hat{u}^k_T|\geq \delta] 
\leq  \sup_{\theta^{\ast}\in\mathcal{K}}\sup_{T\geq T_0}P_{\theta^{\ast}}[\psi(\hat{u}^k_T)\geq \inf_{R_1\geq |y-\hat{u}^k_T|\geq \delta} \psi(y),|\hat{u}^k_T|\leq R_2]+\epsilon/5. 
\end{equation*}
Moreover, by Lemma \ref{G-est-lemma}, there exists $\eta>0$ such that
\begin{equation*}
\sup_{\theta^{\ast}\in\mathcal{K}}P_{\theta^{\ast}}[\inf_{R_1\geq |z|\geq \delta}G_k(z;\theta^{\ast})\leq 3\eta]<\epsilon/5.
\end{equation*}
Furthermore, by $[A1\mathchar`-p]$ and $[A4]$, there exists $R_3>2R_2$ such that
\begin{eqnarray}
\sup_{\theta^{\ast}\in\mathcal{K}}\sup_{T\geq T_0}P_{\theta^{\ast}}[|\tilde{u}^k_T-\hat{u}^k_T|\geq \delta] 
&\leq & \sup_{\theta^{\ast}\in\mathcal{K}}\sup_{T\geq T_0}P_{\theta^{\ast}}\bigg[\int_{|u_k|\leq R_3}w_k(u_k-\hat{u}^k_T)Z^k_T(u_k;\underline{\tilde{\theta}}_{k-1},\theta^{\ast}_k,\overline{\theta^{\star}}_{k+1})\pi_k(\theta^{\ast}+a^k_Tu_k)du_k \nonumber \\
&& \quad \geq \inf_{R_1 \geq |y-\hat{u}^k_T|\geq \delta}\int_{|u_k|\leq R_3}w_k(u_k-y)Z^k_T(u_k;\underline{\tilde{\theta}}_{k-1},\theta^{\ast}_k,\overline{\theta^{\star}}_{k+1})\pi_k(\theta^{\ast}+a^k_Tu_k)du_k \nonumber \\
&&-\eta\inf_{\theta^{\ast}\in\mathcal{K}}\pi_k(\theta^{\ast}_k),|\hat{u}^k_T|\leq R_2\bigg]+\frac{2}{5}\epsilon, \nonumber 
\end{eqnarray}
and 
\begin{equation*}
\sup_{\theta^{\ast}\in\mathcal{K}}P_{\theta^{\ast}}\bigg[\sup_{|z|\leq R_1} \int_{|u_k|>R_3/2}w_k(u_k-z)\exp\bigg(-\frac{1}{2}\Gamma^k(\overline{\theta^{\star}}_{k+1},\theta^{\ast})[u_k,u_k]\bigg)du_k>\frac{\eta}{2}\bigg]<\frac{\epsilon}{5}.
\end{equation*}
Then continuity of $\pi_k$, $[A6]$, Lemmas \ref{logZ-lim-lemma} and \ref{G-est-lemma} yield 
\begin{eqnarray}
&&\sup_{\theta^{\ast}\in\mathcal{K}}\sup_{T\geq T_0}P_{\theta^{\ast}}[|\tilde{u}^k_T-\hat{u}^k_T|\geq \delta] \nonumber \\
\begin{discuss}
&\leq &{\colorr \sup_{\theta^{\ast}\in\mathcal{K},T\geq T_0}P_{\theta^{\ast}}\bigg[\int_{|u_k|\leq R_3}w_k(u_k-\hat{u}^k_T)\exp(-\Gamma^k[u_k-\hat{u}^k_T,u_k-\hat{u}^k_T]/2)\pi_k(\theta^{\ast}_k+a^k_Tu_k)du_k } & \nonumber \\
&&{\colorr \geq \inf_{R_1\geq |y-\hat{u}^k_T|\geq \delta}\int_{|u_k|\leq R_3}w_k(u_k-y)\exp(-\Gamma^k[u_k-\hat{u}^k_T,u_k-\hat{u}^k_T]/2)\pi_k(\theta^{\ast}_k+a^k_Tu_k)du_k -2\eta,|\hat{u}^k_T|\leq R_2\bigg]+\frac{3}{5}\epsilon } & \nonumber \\
\end{discuss}
&\leq &\sup_{\theta^{\ast}\in\mathcal{K}}P_{\theta^{\ast}}\bigg[\int_{|u_k|\leq R_2+R_3}w_k(u_k)\exp\bigg(-\frac{1}{2}\Gamma^k(\overline{\theta^{\star}}_{k+1},\theta^{\ast})[u_k,u_k]\bigg)du_k \nonumber \\
&&\quad \geq \inf_{R_1\geq |z|\geq \delta}\int_{|u_k|\leq R_3/2}w_k(u_k-z)\exp\bigg(-\frac{1}{2}\Gamma^k(\overline{\theta^{\star}}_{k+1},\theta^{\ast})[u_k,u_k]\bigg)du_k-2\eta\bigg]+\frac{3}{5}\epsilon \nonumber \\
&\leq & \sup_{\theta^{\ast}\in\mathcal{K}}P_{\theta^{\ast}}\bigg[\int_{\mathbb{R}^{d_k}}w_k(u_k)\exp\bigg(-\frac{1}{2}\Gamma^k(\overline{\theta^{\star}}_{k+1},\theta^{\ast})[u_k,u_k]\bigg)du_k \nonumber \\
&&\geq \inf_{R_1\geq |z|\geq \delta} \int_{\mathbb{R}^{d_k}}w_k(u_k-z)\exp\bigg(-\frac{1}{2}\Gamma^k(\overline{\theta^{\star}}_{k+1},\theta^{\ast})[u_k,u_k]\bigg)du_k-3\eta\bigg]+\frac{4}{5}\epsilon \nonumber \\
&\leq & \sup_{\theta^{\ast}\in\mathcal{K}}P_{\theta^{\ast}}\big[3\eta \geq \inf_{R_1\geq |z|\geq \delta}G_k(z)\big]+4\epsilon/5<\epsilon. \nonumber 
\end{eqnarray}
\qed
\newpage
\begin{discuss}
{\colorr
\begin{eqnarray}
&&\int_{|u_k|\leq R_1}w_k(u_k-y)Z^k_T(u_k)\pi_k(\theta^{\ast}_k+a^k_Tu_k)du_k \nonumber \\
&\sim &\pi_k(\theta^{\ast})\exp\bigg(\frac{1}{2}\Gamma^k(\overline{\theta^{\ast}}_{k+1},\theta^{\ast})[\hat{u}^k_T,\hat{u}^k_T]\bigg)
\int_{|u_k|\leq R_1}w_k(u_k-y)\exp\bigg(-\frac{1}{2}\Gamma^k(\overline{\theta^{\ast}}_{k+1},\theta^{\ast})[u_k-\hat{u}^k_T,u_k-\hat{u}^k_T]\bigg)du_k \nonumber \\
&=&\pi_k(\theta^{\ast})\exp\bigg(\frac{1}{2}\Gamma^k(\overline{\theta^{\ast}}_{k+1},\theta^{\ast})[\hat{u}^k_T,\hat{u}^k_T]\bigg)
\int_{|u_k+\hat{u}^k_T|\leq R_1}w_k(u_k+\hat{u}^k_T-y)\exp\bigg(-\frac{1}{2}\Gamma^k(\overline{\theta^{\ast}}_{k+1},\theta^{\ast})[u_k,u_k]\bigg)du_k. \nonumber 
\end{eqnarray}
}
\end{discuss}

\noindent
{\bf Proof of Theorem \ref{main2}.}

Since $w\in{\bf W}$, we have (\ref{wk-large-est}).
Then there exists $q'\in (q,L_2\eta)$ such that	
\begin{eqnarray}
&&\sup_{\theta^{\ast}\in \mathcal{K},T\geq T_0}P_{\theta^{\ast}}[|\tilde{u}^k_T|\geq r] \nonumber \\
&\leq &\sup_{\theta^{\ast}\in \mathcal{K},T\geq T_0}P_{\theta^{\ast}}\bigg[\inf_{|z|\geq r}\int_{\{|u_k|\leq r^{\eta}\}\cap U^k_T(\theta^{\ast})}w_k(u_k-z)Z^k_T(u_k;\underline{\tilde{\theta}}_{k-1},\theta^{\ast}_k,\overline{\theta^{\star}}_{k+1})\pi_k(\theta^{\ast}_k+a^k_Tu_k)du_k \nonumber \\
&&\leq \int_{\{|u_k|\leq r^{\eta}\}\cap U^k_T(\theta^{\ast})}w_k(u_k)Z^k_T(u_k;\underline{\tilde{\theta}}_{k-1},\theta^{\ast}_k,\overline{\theta^{\star}}_{k+1})\pi_k(\theta^{\ast}_k+a^k_Tu_k)du_k+e^{-r^{\eta}/2}\bigg]+\frac{C}{r^{q'}} \nonumber \\
&\leq &\sup_{\theta^{\ast}\in \mathcal{K},T\geq T_0}P_{\theta^{\ast}}\bigg[\bigg(\inf_{\theta_k}\pi_k(\theta_k)\bigg)e^{-r^{\eta}/3}\inf_{|z|\geq r}\int_{|u_k|\leq \delta_1}(w_k(u_k-z)-w_k(u_k))du_k\leq e^{-r^{\eta}/2}\bigg]+\frac{C}{r^{q'}}+\frac{3^{L_1}C_{\delta_1}}{r^{L_1\eta}} \nonumber \\
&\leq &\frac{C}{r^{q'}}+\frac{3^{L_1}C_{\delta_1}}{r^{L_1\eta}} \nonumber
\end{eqnarray}
for sufficiently large $r$.

Hence we have
\begin{equation*}
\sup_{\theta^{\ast}\in\mathcal{K},T\geq T_0}E_{\theta^{\ast}}[|\tilde{u}^k_T|^q]\leq \int^{\infty}_0qr^{q-1}\sup_{\theta^{\ast}\in\mathcal{K},T\geq T_0}P_{\theta^{\ast}}[|\tilde{u}^k_T|\geq r]dr<\infty.
\end{equation*}
\qed

\noindent
{\bf Proof of Theorem \ref{ergodic-diffusion-theorem}.}

We apply Corollary \ref{main-cor} 2. with $a^1_n=\sqrt{n}$, $a^2_n=\sqrt{nh_n}$.
Theorem 13 in Yoshida \cite{yos05} yields
\begin{equation*}
(\sqrt{n}(\hat{\theta}^1_n-\theta^{\ast}_1),\sqrt{nh_n}(\hat{\theta}^2_n-\theta^{\ast}_2))\to^d (\zeta_1,\zeta_2)
\end{equation*} 
as $n\to\infty$. Hence we have $[A6]$. 
Let $\eta$ be the variable in $[D5]$. then Condition $[C1\mathchar`-\eta]$ holds.
Moreover, by Lemmas 6 and 7, nonrandomness of $\Gamma^1$ and $\Gamma^2$ and a similar argument to Lemma 9, 
we have $[A2]$.
Inequality (22) and a similar argument to the proof of (24) yield $[C3\mathchar`-q,\eta]$ for any $q>1$.
Furthermore, by $[D3]$ and $[D4]$, we have $[A4]$.

Therefore, it is sufficient to show that $[A3]$ and $[C2\mathchar`-q,\eta]$ hold for any $q>1$.

Let $p>d_1+d_2$, then by Sobolev's inequalities, we have
\begin{eqnarray}
E\bigg[\bigg(\frac{1}{\sqrt{n}}\sup_{\theta\in\Theta}|\partial_{\theta_1}\partial_{\theta_2}H_n(\theta)|\bigg)^p \bigg]
&\leq &C\sup_{\theta\in\Theta}E\bigg[\bigg(\frac{1}{\sqrt{n}}|\partial_{\theta_1}\partial_{\theta_2}H_n(\theta)|\bigg)^p \bigg] 
+C\sup_{\theta\in\Theta}E\bigg[\bigg(\frac{1}{\sqrt{n}}|\partial^2_{\theta_1}\partial_{\theta_2}H_n(\theta)|\bigg)^p \bigg]\nonumber \\
&&+C\sup_{\theta\in\Theta}E\bigg[\bigg(\frac{1}{\sqrt{n}}|\partial_{\theta_1}\partial^2_{\theta_2}H_n(\theta)|\bigg)^p \bigg]. \nonumber
\end{eqnarray}

Since
\begin{equation*}
\partial_{\theta_1}\partial_{\theta_2}H_n(\theta)=\sum_{i=1}^n\partial_{\theta_1}(B(X_{(i-1)h_n},\theta_1)^{-1})[\partial_{\theta_2}a(X_{(i-1)h_n},\theta_2),X_{ih_n}-X_{(i-1)h_n}-ha(X_{(i-1)h_n},\theta_2)],
\end{equation*}
we have $E[(n^{-1/2}\sup_{\theta\in\Theta}|\partial_{\theta_1}\partial_{\theta_2}H_n(\theta)|)^p]\to 0$
by the Burkholder-Davis-Gundy inequality and $nh_n^2\to 0$.
Hence we obtain $[A3]$.

On the other hand, we have
\begin{eqnarray}
H_n(\theta^{\ast}_1+n^{-1/2}u_1,\theta^{\star}_2)-H_n(\theta^{\ast}_1,\theta^{\star}_2)
&=&n^{-1/2}\partial_{\theta_1}H_n(\theta^{\ast}_1,\theta^{\star}_2)[u_1]+\partial^2_{\theta_1}H_n(\theta^{\ast}_1,\theta^{\star}_2)[u_1^{\otimes 2}]/(2n) \nonumber \\
&&+\frac{1}{n\sqrt{n}}\int^1_0\frac{(1-t)^2}{2}\partial^3_{\theta_1}H_n(\theta^{\ast}_1+t\frac{u_1}{\sqrt{n}},\theta^{\star}_2)[u_1^{\otimes 3}]dt, \nonumber
\end{eqnarray}
and 
\begin{eqnarray}
H_n(\tilde{\theta}^1_n,\theta^{\ast}_2+(nh_n)^{-1/2}u_2)-H_n(\tilde{\theta}^1_n,\theta^{\ast}_2)
&=&(nh_n)^{-1/2}\partial_{\theta_2}H_n(\tilde{\theta}^1_n, \theta^{\ast}_2)[u_2]+\partial^2_{\theta_2}H_n(\tilde{\theta}^1_n,\theta^{\ast}_2)[u_2^{\otimes 2}]/(2nh_n) \nonumber \\
&&+(nh_n)^{-3/2}\int^1_0\frac{(1-t)^2}{2}\partial^3_{\theta_2}H_n(\tilde{\theta}^1_n,\theta^{\ast}_2+t\frac{u_2}{\sqrt{nh_n}})[u_2^{\otimes 3}]dt \nonumber
\end{eqnarray}
if $\{\theta^{\ast}_1+tn^{-1/2}u_1\}_{0\leq t\leq 1}\subset \Theta_1$ and $\{\theta^{\ast}_2+t(nh_n)^{-1/2}u_2\}_{0\leq t\leq 1}\subset \Theta_2$.

Hence for any $L>0$, there exists $C>0$ and $n_0\in\mathbb{N}$ such that
\begin{eqnarray}
\sup_{n\geq n_0}P\bigg[\sup_{|u_1|\leq R}|H_n(\theta^{\ast}_1+n^{-1/2}u_1,\theta^{\star}_2)-H_n(\theta^{\ast}_1,\theta^{\star}_2)|\geq r\bigg] \leq \frac{C}{r^L} \quad (r>0),  \nonumber 
\end{eqnarray}
by Lemmas 3, 6 and 7 in Yoshida \cite{yos05}. Similarly, we have
\begin{eqnarray}
\sup_{n\geq n_0}P\bigg[\sup_{|u_2|\leq R}|H_n(\tilde{\theta}^1_n,\theta^{\ast}_2+(nh_n)^{-1/2}u_2)-H_n(\tilde{\theta}^1_n,\theta^{\ast}_2)|\geq r\bigg] \leq \frac{C}{r^L} \quad (r>0).  \nonumber 
\end{eqnarray}
These estimates yield $[C2\mathchar`-q,\eta]$ for any $q>0$.
\qed

\bibliographystyle{plain}

\begin{discuss}
{\colorg
やること
\begin{itemize}
\item Theorem 3の証明で$[A2]$をチェックしているところもう一度確認
\item simaltaneous estimationもやらないとだめ. $K$の数は$H_T$の収束レートから自動的に決まる.
\item $\mathcal{X}_T$で考えなくていいのか? 1つにまとめられるか. 
\item Andersen's lemmaあっているかチェック. strictly convexのdefをチェック.
\item 今回の結果がほんとにMasudaに適用できるかチェックする.
\end{itemize}
余裕があったらやること
\begin{itemize}
\item $\overline{\theta^{\star}}_{k+1}$を$T$に依存するr.v.に変えるか
\item $H$の微分可能性を仮定しないで, 
$\log Z(u)-\Gamma[\hat{u}_T,\hat{u}_T]/2+\Gamma[u_k-\hat{u}^k_T,u_k-\hat{u}^k_T]\to^p 0$
を示すための十分条件を考えるか.
\end{itemize}
}
\end{discuss}
\end{document}